\documentclass[12]{amsart}
\usepackage{amsmath,amssymb,amsthm,color,enumerate,comment,centernot,enumitem,url,cite}
\usepackage{graphicx,relsize,bm}
\usepackage{mathtools}
\usepackage{array}

\makeatletter
\newcommand{\tpmod}[1]{{\@displayfalse\pmod{#1}}}
\makeatother

\newtheorem{thm}{Theorem}[section]
\newtheorem{lemma}[thm]{Lemma}

\newtheorem{prop}[thm]{Proposition}
\newtheorem{cor}[thm]{Corollary}

\theoremstyle{remark}

\theoremstyle{definition}
    \newtheorem{defn}[thm]{Definition}

\newtheorem{rem}[thm]{Remark}

\theoremstyle{THM}

\newcommand{\abs}[1]{\left|{#1}\right|}

\def\FF {{\mathcal F}}
\def\GG {{\mathcal G}}

\def\Z {{\mathbb Z}}

\def\Q {{\mathbb Q}}

\def\GG {{\mathcal G}}

\def\F {{\mathbb F}}

\def\Z {{\mathbb Z}}
\def\Q {{\mathbb Q}}

\def\Gal{{\mbox{{\rm{Gal}}}}}

\makeatletter
\@namedef{subjclassname@2020}{%
  \textup{2020} Mathematics Subject Classification}
\makeatother

\def\red#1 {\textcolor{red}{#1 }}
\def\blue#1 {\textcolor{blue}{#1 }}

\numberwithin{equation}{section}

\def\Z {{\mathbb Z}}

\newcommand{\Mod}[1]{\ (\mathrm{mod}\enspace #1)}
\newcommand{\mmod}[1]{\ \mathrm{mod}\enspace #1}
\begin{document}

\title[The monogenicity of certain reciprocal quintinomials]{The monogenicity and Galois groups of certain reciprocal quintinomials}


\author{Lenny Jones}
\address{Professor Emeritus, Department of Mathematics, Shippensburg University, Shippensburg, Pennsylvania 17257, USA}
\email[Lenny~Jones]{doctorlennyjones@gmail.com}

\date{\today}

\begin{abstract}
We say that a monic polynomial $f(x)\in \Z[x]$ is \emph{monogenic} if $f(x)$ is irreducible over $\Q$ and $\{1,\theta,\theta^2,\ldots ,\theta^{\deg(f)-1}\}$ is a basis for $\Z_K$, the ring of integers of $K=\Q(\theta)$, where $f(\theta)=0$. For $n\ge 2$, we define the reciprocal quintinomial
\[\FF_{n,A,B}(x):=x^{2^n}+Ax^{3\cdot 2^{n-2}}+Bx^{2^{n-1}}+Ax^{2^{n-2}}+1\in \Z[x].\]
In this article, we extend our previous work on the monogenicity of $\FF_{n,A,B}(x)$ to treat the specific previously-unaddressed situation of $A\equiv B\equiv 1\pmod{4}$. Moreover, we determine the Galois group over $\Q$ of $\FF_{n,A,B}(x)$ in special cases.
\end{abstract}


\maketitle
\section{Introduction}\label{Section:Intro}

 We say that a monic polynomial $f(x)\in \Z[x]$ is \emph{monogenic} if $f(x)$ is irreducible over $\Q$ and $\{1,\theta,\theta^2,\ldots ,\theta^{\deg(f)-1}\}$ is a basis for $\Z_K$, the ring of integers of $K=\Q(\theta)$, where $f(\theta)=0$. Hence, $f(x)$ is monogenic if and only if  $\Z_K=\Z[\theta]$. For the minimal polynomial $f(x)$ of an algebraic integer $\theta$ over $\Q$, it is well known \cite{Cohen} that
\begin{equation} \label{Eq:Dis-Dis}
\Delta(f)=\left[\Z_K:\Z[\theta]\right]^2\Delta(K),
\end{equation}
where $\Delta(f)$ and $\Delta(K)$ are the discriminants over $\Q$ of $f(x)$ and the number field $K$, respectively.
Thus, from \eqref{Eq:Dis-Dis}, $f(x)$ is monogenic if and only if $\Delta(f)=\Delta(K)$.

  Throughout this article, for $A,B,n\in \Z$ with $AB\ne 0$ and $n\ge 2$, we let:
  \begin{align}\label{Eq:Defs}
  \begin{split}
  \FF_{n,A,B}(x)&=x^{2^n}+Ax^{3\cdot 2^{n-2}}+Bx^{2^{n-1}}+Ax^{2^{n-2}}+1,\\
  W_1&=B+2-2A, \quad W_2=B+2+2A, \quad W_3=A^2-4B+8,\\
  P&=\gcd(W_1,W_3), \quad Q=\gcd(W_1,W_2), \quad R=\gcd(W_2,W_3),\\
  D_n& \ \mbox{denote the dihedral group of order $2n$},\\
  C_n& \ \mbox{denote the cyclic group of order $n$.}
  \end{split}
  \end{align}

The following theorem was proven in \cite{JonesBAMS}:
 \begin{thm}\label{Thm:BAMS} If $W_1W_2W_3$ is squarefree and
 \[(A \mmod{4}, \ B \mmod{4})\in \{(1,3),(3,1),(3,3)\},\] 
 then $\FF_{n,A,B}(x)$
 is monogenic for all $n\ge 2$.
 \end{thm}

  It is the goal of this article to extend Theorem \ref{Thm:BAMS} by providing an investigation of the monogenicity of $\FF_{n,A,B}(x)$ in the specific previously-unaddressed situation of $A\equiv B\equiv 1\pmod{4}$. Moreover, we determine the Galois group of $\FF_{n,A,B}(x)$ in some special cases.
  More precisely, we prove
\begin{thm}\label{Thm:Main1}
  Assuming the notation of \eqref{Eq:Defs}, suppose that $A\equiv B\equiv 1 \pmod{4}$. Then
\begin{enumerate}
\item \label{Main1 I1} $\FF_{2,A,B}(x)$ is irreducible over $\Q$ and
\[\Gal(\FF_{2,A,B})\simeq \left\{\begin{array}{l}
  C_4 \quad  \mbox{if and only if} \quad W_1W_2W_3 \ \mbox{is a square,}\\[.5em]
  D_4 \quad \mbox{if and only if} \quad  W_1W_2W_3 \ \mbox{is not a square.}
\end{array}\right.\]
\item \label{Main1 I2} $\FF_{2,A,B}(x)$ is monogenic if and only if $W_1$, $W_2$ and $W_3$ are squarefree.
\item \label{Main1 I3} There exist infinitely many pairs $(A,B)$ such that $\FF_{2,A,B}(x)$ is monogenic with $\Gal(\FF_{2,A,B})\simeq D_4$. Furthermore, for any two such pairs $(A_i,B_i)\ne (A_j,B_j)$ where $\FF_{2,A_i,B_i}(\theta_i)=\FF_{2,A_j,B_j}(\theta_j)=0$, we have that $\Q(\theta_i)\ne \Q(\theta_j)$.
\item \label{Main1 I4} $\FF_{2,A,B}(x)$ is monogenic with $\Gal(\FF_{2,A,B})\simeq C_4$ if and only if $A=B=1$.
\item \label{Main1 I5} $\FF_{3,A,B}(x)$ is reducible over $\Q$ if and only if
\begin{align*}
A&=4t-4s^2-4s+1 \quad \mbox{and}\\
B&\in \{4t^2+4t-8s^2-8s+1, \ 4t^2+4t+8s^2+8s+5\},
\end{align*}
for some $s,t\in \Z$.
\item \label{Main1 I6} There exist infinitely many values of $A$ such that $\Gal(\FF_{3,A,A})$ is isomorphic to the wreath product $C_2^2\wr C_2$.
\item \label{Main1 I7}  When $n\ge 3$, $\FF_{n,A,A}(x)$ is irreducible over $\Q$ if and only if $A\ne 1$.
\item \label{Main1 I8} $\FF_{n,A,A}(x)$ is never monogenic when $n\ge 3$.
\end{enumerate}
\end{thm}

\section{Preliminaries}\label{Section:Prelim}
The first result will be useful in the proof of item \eqref{Main1 I4} of Theorem \ref{Thm:Main1}.
\begin{prop}{\rm \cite{Koshy}}\label{Prop:LF}
  For  $N\in \Z$, let ${\mathfrak L}_N$ and ${\mathfrak F}_N$ denote, respectively, the $N$th Lucas and $N$th Fibonacci numbers, where ${\mathfrak L}_0=2$ and ${\mathfrak F}_0=0$. Then
  \begin{enumerate}
  \item \label{LF0} ${\mathfrak L}_{-N}=(-1)^N{\mathfrak L}_{N}$,
    \item \label{LF1} $5{\mathfrak F}_N=2{\mathfrak L}_{N+1}-{\mathfrak L}_N$.
    \item \label{LF2} ${\mathfrak L}_{2N}+(-1)^N2={\mathfrak L}_N^2$.
  \end{enumerate}
\end{prop}

The proof of the next proposition can be found in \cite{JonesBAMS}.
\begin{prop}\label{Eq:Delta(f)}
  $\Delta(\FF_{n,A,B})=2^{2^n(n-2)}\left(W_1W_2W_3^2\right)^{2^{n-2}}$.
\end{prop}

The following proposition, which follows from a generalization of a theorem of Capelli, is a special case of the results in \cite{NG}, and gives simple necessary and sufficient conditions for the irreducibility of polynomials of the form $w(x^{2^k})\in \Z[x]$, when $w(x)$ is monic and irreducible.
\begin{prop}{\rm \cite{NG}}\label{Prop:NG}
  Let $w(x)\in \Z[x]$ be monic and irreducible, with $\deg(w)=m$. Then $w\left(x^{2^k}\right)$ is reducible if and only if there exist $S_0(x),S_1(x)\in \Z[x]$ such that either
  \begin{equation}\label{Eq:NGC1}
  (-1)^mw(x)=\left(S_0(x)\right)^2-x\left(S_1(x)\right)^2,
  \end{equation}
  or
  \begin{equation}\label{Eq:NGC2}
  k\ge 2 \mbox{ and } w\left(x^2\right)=\left(S_0(x)\right)^2-x\left(S_1(x)\right)^2.
  \end{equation}
\end{prop}

The following theorem, known as \emph{Dedekind's Index Criterion}, or simply \emph{Dedekind's Criterion} if the context is clear, is a standard tool used in determining the monogenicity of a polynomial.
\begin{thm}[Dedekind \cite{Cohen}]\label{Thm:Dedekind}
Let $K=\Q(\theta)$ be a number field, $T(x)\in \Z[x]$ the monic minimal polynomial of $\theta$, and $\Z_K$ the ring of integers of $K$. Let $q$ be a prime number and let $\overline{ * }$ denote reduction of $*$ modulo $q$ (in $\Z$, $\Z[x]$ or $\Z[\theta]$). Let
\[\overline{T}(x)=\prod_{i=1}^k\overline{\tau_i}(x)^{e_i}\]
be the factorization of $T(x)$ modulo $q$ in $\F_q[x]$, and set
\[h_1(x)=\prod_{i=1}^k\tau_i(x),\]
where the $\tau_i(x)\in \Z[x]$ are arbitrary monic lifts of the $\overline{\tau_i}(x)$. Let $h_2(x)\in \Z[x]$ be a monic lift of $\overline{T}(x)/\overline{h_1}(x)$ and set
\[F(x)=\dfrac{h_1(x)h_2(x)-T(x)}{q}\in \Z[x].\]
Then
\[\left[\Z_K:\Z[\theta]\right]\not \equiv 0 \pmod{q} \Longleftrightarrow \gcd\left(\overline{F},\overline{h_1},\overline{h_2}\right)=1 \mbox{ in } \F_q[x].\]
\end{thm}

\begin{thm}\label{Thm:Pasten}
   Let $G(t)\in \Z[t]$, and suppose that $G(t)$ factors into a product of distinct non-constant polynomials $\gamma_i(t)\in \Z[x]$ that are irreducible over $\Z$, such that the degree of each $\gamma_i(t)$ is at most 3.   Define
   \[N_G\left(X\right)=\abs{\left\{p\le X : p \mbox{ is prime and } G(p) \mbox{ is squarefree}\right\}}.\]
      Then,
   \begin{equation}\label{Eq:NG}
   N_G(X)\sim C_G\dfrac{X}{\log(X)},
   \end{equation}
   where
   \begin{equation}\label{Eq:CG}
   C_G=\prod_{\ell  \mbox{ \rm {\tiny prime}}}\left(1-\dfrac{\rho_G\left(\ell^2\right)}{\ell(\ell-1)}\right)
   \end{equation}
   and $\rho_G\left(\ell^2\right)$ is the number of $z\in \left(\Z/\ell^2\Z\right)^{*}$ such that $G(z)\equiv 0 \pmod{\ell^2}$.
 \end{thm}

\begin{rem}
  Theorem \ref{Thm:Pasten} follows from work of Helfgott, Hooley and Pasten \cite{Helfgott-cubic,Hooley-book,Pasten}. For more details, see the discussion following \cite[Theorem 2.11]{JonesNYJM}.
\end{rem}

\begin{defn}\label{Def:Obstruction}
 In the context of Theorem \ref{Thm:Pasten}, for $G(t)\in \Z[t]$ and a prime $\ell$, if $G(z)\equiv 0 \pmod{\ell^2}$ for all $z\in \left(\Z/\ell^2\Z\right)^{*}$, we say that $G(t)$ has a \emph{local obstruction at $\ell$}. A polynomial $G(t)\in \Z[t]$ is said to have \emph{no local obstructions}, if for every prime $\ell$ there exists some $z\in \left(\Z/\ell^2\Z\right)^{*}$ such that $G(z)\not \equiv 0 \pmod{\ell^2}$.
\end{defn}

Note that $C_G>0$ in \eqref{Eq:CG} if and only if $G(t)$ has no local obstructions. Consequently, it follows that $N_G(X)\to \infty$ as $X\to \infty$ in \eqref{Eq:NG}, when $G(t)$ has no local obstructions. Hence, we have the
following immediate corollary of Theorem \ref{Thm:Pasten}.
\begin{cor}\label{Cor:Squarefree}
 Let $G(t)\in \Z[t]$, and suppose that $G(t)$ factors into a product of distinct non-constant polynomials $\gamma_i(t)\in \Z[x]$ that are irreducible over $\Z$, such that the degree of each $\gamma_i(t)$ is at most $3$. To avoid the situation when $C_G=0$ (in \eqref{Eq:CG}), we suppose further that $G(t)$ has no local obstructions.
  Then there exist infinitely many primes $p$ such that $G(p)$ is squarefree.
\end{cor}
The following lemma, which generalizes a discussion found in \cite{JonesBAMS}, will be useful in the proof of item \eqref{Main1 I4} of Theorem \ref{Thm:Main1}.
\begin{lemma}\label{Lem:ObstructionCheck}
  Let $G(t)\in \Z[t]$ with $\deg(G)=N$, and suppose that $G(t)$ factors into a product of distinct non-constant polynomials that are irreducible over $\Z$, such that the degree of each factor is at most $3$. If $G(t)$ has an obstruction at the prime $\ell$, then $\ell\le (N_{\ell}+2)/2$, where $N_{\ell}$ is the number of not-necessarily distinct non-constant linear factors of $G(t)$ in $\F_{\ell}[t]$.
\end{lemma}
\begin{proof}
Since no factors of $G(t)$ in $\Z[t]$ are constant, we can assume that the content of every factor of $G(t)$ is 1.
Furthermore, since a nonlinear irreducible factor of $G(t) \Mod{\ell}$ never has a zero in $\left(\Z/\ell^2\Z\right)^*$, we can also assume, without loss of generality, that $G(t)$ factors completely into $N$, not-necessarily distinct, non-constant linear factors in $\Z[t]$. Thus,
\begin{equation}\label{Eq:G(t)}
 G(t)\equiv c\prod_{j=0}^{\ell-1}(t-j)^{e_j} \pmod{\ell},
 \end{equation} where $c\not \equiv 0 \pmod{\ell}$, $e_j\ge 0$ for each $j$ and $N=\sum_{j=0}^{\ell-1}e_j$.
   Observe that if $e_j=0$ for some $j\ne 0$ in \eqref{Eq:G(t)}, then $G(j)\not \equiv 0 \Mod{\ell^2}$, contradicting the fact that $G(t)$ has an obstruction at the prime $\ell$. If $e_j=1$ for some $j\ne 0$ in \eqref{Eq:G(t)}, then the zero $j$ of $x-j \Mod{\ell}$ lifts to the unique zero $j$ of $x-j \Mod{\ell^2}$. Thus, $G(j+\ell)\not \equiv 0 \Mod{\ell^2}$, again contradicting the fact that $G(t)$ has an obstruction at the prime $\ell$. Hence, $e_j\ge 2$ for all $j\in \{1,2,\ldots,\ell-1\}$. Assume, by way of contradiction, that $\ell> (N+2)/2$. Then
  \[2(\ell-1)>N=\sum_{j=0}^{\ell-1}e_j=e_0+\sum_{j=1}^{\ell-1}e_j\ge e_0+2(\ell-1),\] which is impossible, and the proof is complete.
\end{proof}

The next theorem follows from \cite{AwtreyPatane}.
\begin{thm} \label{Thm:AP} Assuming the notation of \eqref{Eq:Defs}, suppose that $A\equiv B\equiv 1 \pmod{4}$. If $\FF_{3,A,B}(x)$ is irreducible over $\Q$, then
$\Gal(\FF_{3,A,B})\simeq C_2^2\wr C_2$ if and only if
\[\mbox{none of } \ W_1,\ W_2,\ W_1W_2,\ W_1W_3,\ W_2W_3 \ \mbox{ and }\ W_1W_2W_3 \ \mbox{ is a square,}\] 
 \end{thm}

\section{The Proof of Theorem \ref{Thm:Main1}}\label{Section:Main1 Proof}
The following lemma will be useful in the proof of item \eqref{Main1 I8} of Theorem \ref{Thm:Main1}
\begin{lemma}\label{Lem:Inductivemonogenicity}
  If $\FF_{n,A,B}(x)$ is monogenic for some $n\ge 3$, then $\FF_{n-1,A,B}(x)$ is monogenic.
\end{lemma}
\begin{proof}
Let $\Z_{K_n}$ be the ring of integers of $K_n=\Q(\theta)$, where $\FF_{n,A,B}(\theta)=0$. Then  $\{1,\theta,\theta^2,\ldots ,\theta^{2^n-1}\}$ is a basis for $\Z_{K_n}$ since $\FF_{n,A,B}(x)$ is monogenic. Observe that $\FF_{n,A,B}(x)=\FF_{n-1,A,B}(x^2)$. It follows that $\FF_{n-1,A,B}(x)$ is the minimal polynomial of $\theta^{1/2}$ and  $\{1,\theta^{1/2},\theta,\theta^{3/2},\ldots ,\theta^{(2^{n-1}-1)/2}\}$ is a power basis for $\Z_{K_{n-1}}$, the ring of integers of $K_{n-1}=\Q(\theta^{1/2})$.
\end{proof}
\begin{proof}[Proof of Theorem \ref{Thm:Main1}]
  For item \eqref{Main1 I1}, since $A\equiv B\equiv 1 \pmod{4}$, it is easy to verify that
  \[W_1W_2\equiv W_3\equiv 5\pmod{8},\] which implies that neither $W_1W_2$ nor $W_3$ is a square. Hence, item \eqref{Main1 I1} follows from \cite{Dickson}.

  For item \eqref{Main1 I2}, let $K=\Q(\theta)$ with ring of integers $\Z_K$, where $\FF_{2,A,B}(\theta)=0$. Suppose first that $\FF_{2,A,B}(x)$ is monogenic and assume, by way of contradiction, that $W_i$ is not squarefree for some $i$. Since $A\equiv B\equiv 1 \pmod{4}$, note that $2\nmid W_i$.

  We begin with $i=1$, and suppose that $q$ is a prime such that $q^2\mid W_1$. Using $q$, we apply Theorem \ref{Thm:Dedekind} with $T(x):=\FF_{n,A,B}(x)$. Then $B\equiv 2A-2\pmod{q}$, so that
  \[T(x)\equiv (x+1)^2g(x) \pmod{q},\] where $g(x)=x^2+(A-2)x+1$.

  If $g(x)$ is irreducible in $\F_q[x]$, then we can let
   \[h_1(x)=(x+1)g(x) \quad \mbox{and} \quad h_2(x)=x+1.\]
   Thus,
     \begin{align*}
     F(x)&=\dfrac{h_1(x)h_2(x)-T(x)}{q}=\dfrac{(x+1)^2g(x)-T(x)}{q}\\
     &=-\left(\dfrac{B+2-2A}{q}\right) x^2\equiv 0 \pmod{q},
     \end{align*}
  which implies that $\gcd(\overline{F},\overline{h_1},\overline{h_2})\ne 1$. Hence, $q\mid [\Z_K:\Z[\theta]]$ by Theorem \ref{Thm:Dedekind}, contradicting the fact that $\FF_{2,A,B}(x)$ is monogenic.

  If $g(x)$ is reducible in $\F_q[x]$, then we can let
   \[h_1(x)=(x+1)(x-r_1)(x-r_2) \quad \mbox{and} \quad h_2(x)=x+1,\] for some $r_1,r_2\in \Z$ with $r_1+r_2\equiv 2-A\pmod{q}$ and $r_1r_2\equiv 1\pmod{q}$.
   Thus,
     \begin{align*}
     F(x)&=\dfrac{h_1(x)h_2(x)-T(x)}{q}\\
     &=\left(\dfrac{2-A-r_1-r_2}{q}\right)x^3+\left(\dfrac{1+r_1r_2-2r_1-2r_2-B}{q}\right)x^2\\
     &\qquad +\left(\dfrac{-A-r_1-r_2+2r_1r_2}{q}\right)x+\dfrac{r_1r_2-1}{q}.
     \end{align*}
     Then, since $F(-1)=-(B+2-2A)/q\equiv 0 \pmod{q}$, it follows that
     \[\gcd(\overline{F},\overline{h_1},\overline{h_2})\equiv 0 \pmod{x+1},\] again contradicting the fact that $\FF_{2,A,B}(x)$ is monogenic. Hence, $W_1$ is squarefree.  The case $i=2$ is similar and we omit the details.

     Now let $i=3$, and let $q$ be a prime with $q^2\mid W_3$. Since $B\equiv (A^2+8)/4 \pmod{q}$, we have that $T(x)\equiv g(x)^2 \pmod{q}$, where
     $g(x)=x^2+(A/2)x+1$.

     If $g(x)$ is irreducible in $\F_q[x]$, we can let
     \[h_1(x)=h_2(x)=x^2+\left(\frac{A+q}{2}\right)x+1\in \Z[x].\]  Thus,
     \[F(x)=\dfrac{h_1(x)h_2(x)-T(x)}{q}=x\left(x^2+\left(\dfrac{\dfrac{A^2-4B+8}{q}+2A+q}{4}\right)x+1\right).\] Hence, since $q^2\mid W_3$, it follows that
     \[\overline{F}(x)=x(x^2+(A/2)x+1)=xg(x),\] and it is easy to see that $\gcd(\overline{F},\overline{h_1},\overline{h_2})\ne 1$. Hence, $q\mid [\Z_K:\Z[\theta]]$ by Theorem \ref{Thm:Dedekind}, again contradicting the fact that $\FF_{2,A,B}(x)$ is monogenic.

     If $g(x)$ is reducible in $\F_q[x]$, then
          \[\overline{T}(x)=\left(x-\dfrac{-A-\sqrt{A^2-16}}{4}\right)^2\left(x-\dfrac{-A+\sqrt{A^2-16}}{4}\right)^2.\]
     Thus, we can let $h_1(x)=h_2(x)=(x-r_1)(x-r_2)$ for some $r_1,r_2\in \Z$ with
     \[r_1\equiv \dfrac{-A-\sqrt{A^2-16}}{4} \pmod{q} \quad \mbox{and}\quad r_2\equiv \dfrac{-A+\sqrt{A^2-16}}{4} \pmod{q}.\]
      Therefore,
     \begin{align*}
     F(x)&=\dfrac{h_1(x)h_2(x)-T(x)}{q}\\
     &=-\left(\dfrac{A+2(r_1+r_2)}{q}\right)x^3+\left(\dfrac{(r_1+r_2)^2-B+2r_1r_2}{q}\right)x^2\\
     &\qquad -\left(\dfrac{A+2r_1r_2(r_1+r_2)}{q}\right)x+\dfrac{(r_1r_2)^2-1}{q},
     \end{align*}
     so that
     \[F(r_1)=-\dfrac{r_1^4+Ar_1^3+Br_1^2+Ar_1+1}{q}=-\dfrac{T(r_1)}{q}.\] Since $T(r_1)\equiv T'(r_1)\equiv 0 \pmod{q}$, it follows by Hensel that $T(r_1)\equiv 0 \pmod{q^2}$. Hence, $\overline{F}(r_1)=0$ so that $\gcd(\overline{F},\overline{h_1},\overline{h_2})\ne 1$ and $q\mid [\Z_K:\Z[\theta]]$ by Theorem \ref{Thm:Dedekind}, supplying the final contradiction in this direction to the fact that $\FF_{2,A,B}(x)$ is monogenic.

     Conversely, suppose now that $W_1$, $W_2$ and $W_3$ are squarefree, and recall the definition of $P$, $Q$ and $R$ in \eqref{Eq:Defs}.
         If $P=Q=R=1$, then $\Delta(\FF_{2,A,B})=\Delta(K)$ from \eqref{Eq:Dis-Dis}, so that  $\FF_{2,A,B}(x)$ is monogenic. That is, we only have to address the primes $q$ dividing $PQR$. Note that $q\ge 3$ since $2\nmid W_1W_2W_3$.

     Suppose first then that $q$ is a prime with $q\mid P$.
     Then $B\equiv 2A-2 \pmod{q}$, since $q\mid W_1$, so that
     \[W_3=A^2-4B+8\equiv (A-4)^2\equiv 0 \pmod{q},\] since $q\mid W_3$.
     Hence, $A\equiv 4 \pmod{q}$ and $B\equiv 6 \pmod{q}$ since $q\ne 2$. Then, with $T(x):=\FF_{2,A,B}(x)$, we have that
     \[\overline{T}(x)=(x+1)^4.\] Therefore, applying Theorem \ref{Thm:Dedekind}, we can let
     \[h_1(x)=x+1 \quad \mbox{and}\quad h_2(x)=(x+1)^3\]
     to get
     \[F(x)=\frac{h_1(x)h_2(x)-T(x)}{q}=-x\left(\left(\frac{A-4}{q}\right)x^2+\left(\frac{B-6}{q}\right)x+\frac{A-4}{q}\right).\] Then
     \[F(-1)=-\frac{B+2-2A}{q}\not \equiv 0 \pmod{q},\]
     since $B+2-2A$ is squarefree. Hence, $\gcd(\overline{F},\overline{h_1},\overline{h_2})=1$, and $q\nmid [\Z_K:\Z[\theta]]$ by Theorem \ref{Thm:Dedekind}.

     Suppose next that $q$ is a prime with $q\mid Q$.  Then, since $q\mid W_1$ and $q\mid W_2$, we have that
       \[B\equiv 2A-2\equiv -2A-2 \pmod{q},\]
     which implies that $q\mid A$ since $q\ne 2$,  and $B\equiv -2 \pmod{q}$. Letting $T(x):=\FF_{2,A,B}(x)$, straightforward calculations yield
     \[\overline{T}(x)=(x-1)^2(x+1)^2.\] Hence, we can let
     \[h_1(x)=h_2(x)=(x-1)(x+1),\] and apply Theorem \ref{Thm:Dedekind} to get that
     \[F(x)=\frac{h_1(x)h_2(x)-T(x)}{q}=-x\left(\left(\frac{A}{q}\right)x^2+\left(\frac{B+2}{q}\right)x+\frac{A}{q}\right).\]
     Thus, since $W_1$ and $W_2$ are squarefree, it follows that
     \[F(1)=-\dfrac{B+2+2A}{q}\not \equiv 0 \pmod{q}\quad \mbox{and}\quad F(-1)=-\dfrac{B+2-2A}{q}\not \equiv 0 \pmod{q}.\] Therefore, $\gcd(\overline{F},\overline{h_1},\overline{h_2})=1$, and $q\nmid [\Z_K:\Z[\theta]]$ by Theorem \ref{Thm:Dedekind}.

The last possibility of $q\mid R$ is similar and we omit the details. Thus, the proof of item \eqref{Main1 I2} is complete.

For item \eqref{Main1 I3}, let $A=8k+1$ and $B=8t+1$, where $k\in \Z$ and $t$ is an indeterminate. Then 
\[W_1=8t-16k+1,\quad W_2=8t+16k+5 \quad \mbox{and} \quad W_3=32t-64k^2-16k-5.\] Let $G(t):=W_1W_2W_3\in \Z[t]$. Then $G(t)$ has no obstructions by Lemma \ref{Lem:ObstructionCheck}. Hence, by Corollary \ref{Cor:Squarefree}, there exist infinitely many primes $p$ such that $G(p)$ is squarefree. Moreover, since $\deg(G(t))=3$, we can assume that there exist infinitely many such primes $p$ with $G(p)>1$. Consequently, $W_1$, $W_2$ and $W_3$ are squarefree for each such prime $p$, and therefore, $\FF_{2,A,B}(x)$ is monogenic by item \eqref{Main1 I2}. Furthermore, for each such prime $p$, it follows that $\Gal(\FF_{2,A,B})\simeq D_4$ by item \eqref{Main1 I2}. To see that these quartic fields are distinct, we assume, by way of contradiction, that there exist primes $p_1\ne p_2$ such that $G(p_i)$ is squarefree, $K_1=\Q(\theta_1)=K_2=\Q(\theta_2)$, where
\[\FF_{2,8k+1,8p_1+1}(\theta_1)=0=\FF_{2,8k+1,8p_2+1}(\theta_2).\] Since $\FF_{2,8k+1,8p_1+1}(x)$ and $\FF_{2,8k+1,8p_1+1}(x)$ are both monogenic, it follows that
\begin{equation}\label{Eq:Diseq}
\Delta(\FF_{2,8k+1,8p_1+1})=\Delta(\FF_{2,8k+1,8p_2+1}).
\end{equation} Recall from \eqref{Eq:Delta(f)} that $\Delta(\FF_{2,8k+1,8t+1})=W_1W_2W_3^2$. 
 Because $G(p_1)$ and $G(p_2)$ are squarefree, it then follows from \eqref{Eq:Diseq} that
 \begin{equation}\label{Eq:equations}
 W_1W_2|_{t=p_1}=W_1W_2|_{t=p_2} \quad \mbox{and} \quad W_3|_{t=p_1}=\pm W_3|_{t=p_2}.
  \end{equation} Maple easily reveals the impossibility in integers of the equations in \eqref{Eq:equations} when $p_1\ne p_2$, which completes the proof of item \eqref{Main1 I3}.

We turn now to item \eqref{Main1 I4}. Since $\FF_{2,1,1}(x)=\Phi_5(x)$, we see that $\FF_{2,1,1}(x)$ is monogenic with $\Gal(\FF_{2,1,1})\simeq C_4$, which proves the direction assuming $A=B=1$.

For the converse, we assume that $\FF_{2,A,B}(x)$ is monogenic with $\Gal(\FF_{2,A,B})\simeq C_4$.
For part of this proof, we use an approach that is a modification of methods employed in \cite[pp. 26--28]{JonesOctics}. Although the arguments are similar, we provide details for the sake of completeness.

Since $\FF_{2,A,B}(x)$ is monogenic with $\Gal(\FF_{2,A,B})\simeq C_4$, it follows from items \eqref{Main1 I2} and \eqref{Main1 I1} of this theorem, respectively, that
\begin{equation}\label{Eq:Wconditions}
W_1,\ W_2\ \mbox{and}\ W_3 \ \mbox{are squarefree, and}\ W_1W_2W_3 \ \mbox{is a square.}
\end{equation}
Since $W_1W_2W_3$ is a square, we have that $W_1W_2$ and $W_3$ are either both positive or both negative.
If
 \[W_1W_2=(B+2)^2-4A^2<0 \quad \mbox{and}\quad W_3=A^2-4B+8<0,\] then
 \[(B+2)^2<4A^2<16B-32,\]
 which yields the contradiction
 \[(B-6)^2=(B+2)^2-16B+32<0.\]
 Hence,
 \begin{equation}\label{Eq:W1W2 W5 pos}
 W_1W_2>0\quad \mbox{and} \quad W_3>0.
 \end{equation}
Then,
  \begin{align}\label{Eq:PQR}
 \begin{split}
   \abs{W_1}&=\abs{B+2-2A}=PQ\\
   \abs{W_2}&=\abs{B+2+2A}=QR\\
  W_3&=A^2-4B+8=PR,
   \end{split}
 \end{align}
 where $PQR$ is squarefree. Thus, either $PQR=1$ or $PQR$ is the product of distinct odd primes. If $PQR=1$, then $P=Q=R=1$, which implies that $A=0$ from \eqref{Eq:PQR}, contradicting the fact that $A\equiv 1 \pmod{4}$. Hence, $PQR>1$ is the product of distinct odd primes.

We claim that $R>1$. To establish this claim, we assume, by way of contradiction, that $R=1$.
  Then, regardless of whether $W_1$ and $W_2$ are both positive or both negative in \eqref{Eq:W1W2 W5 pos},
 it follows from \eqref{Eq:PQR} that $W_1=W_2W_3$.
 Solving this equation in Maple reveals that
 \[y^2=A^4+16A^3+94A^2+304A+225=(A+9)(A+1)(A^2+6A+25),\]
 for some $y\in \Z$. Using the command \[{\bf IntegralQuarticPoints}([1,16,94,304,225],[-1,0]);\] in Magma yields the solutions
 $A\in \{-1,0,-9,-11,4\}$. Since $A\equiv 1 \pmod{4}$, we see that $A=-11$, so that $B\in \{21,31\}$. Since $B\equiv 1 \pmod{4}$, we conclude that the only viable coefficient pair is $(A,B)=(-11,21)$, in which case we have that
 $W_1=B+2-2A=45$, contradicting the fact that $W_1$ is squarefree.

We proceed by providing details first in the situation when $W_1>0$ and $W_2>0$.
Invoking Maple to solve the system \eqref{Eq:PQR}, we get that
\begin{equation}\label{Eq:Main}
P^2Q^2-2PQ^2R+Q^2R^2-32PQ-32QR-16PR+256=0.
\end{equation}
It follows from \eqref{Eq:Main} that
\begin{equation}\label{Eq:Div}
P\mid (QR-16),\quad Q\mid (PR-16)\quad \mbox{and} \quad R\mid (PQ-16).
\end{equation}
  Thus, since $PQR$ is odd and squarefree, we deduce from \eqref{Eq:Div} that $PQR$ divides
  \begin{align*}
  Z:&=\dfrac{(QR-16)(PR-16)(PQ-16)-PQR(PQR-16P-16Q-16R)}{256}\\
  &=PQ+QR+PR-16.
  \end{align*}
  Suppose that $P\ge 3$ and $Q\ge 3$. It is then easy to see that $Z>0$.
   Hence, since $PQR$ divides $Z$, we have that $H:=PQR-Z\le 0$.
However, using Maple, we see that the minimum value of $H$, subject to the constraints $\{P\ge 3,Q\ge 3,R\ge 3\}$, is 16. Thus, we deduce that $P=1$ or $Q=1$.

Letting $P=1$ in \eqref{Eq:Main}, and solving for $Q$ yields
\begin{equation}\label{Q}
Q=\dfrac{4(4R+4\pm \sqrt{R^3-2R^2+65R})}{(R-1)^2}.
\end{equation}
For $Q$ to be a viable solution, we conclude from \eqref{Q} that
\begin{equation}\label{Eq:EQ}
y^2=R^3-2R^2+65R,
\end{equation}
for some integer $y$. Using Sage to find all integral points (with $y\ge 0$) on the elliptic curve \eqref{Eq:EQ} we get
\[(R,y)\in \{(0,0),(1,8),(5,20),(13,52),(16,68),(45,300),(65,520),(1573,62348)\}.\] Since $R\ge 3$ is odd and squarefree, we have that
$R\in \{5,13,65\}$. Plugging these values into \eqref{Q} reveals only the three valid integer solution triples
\begin{equation}\label{Eq:P=1}
(P,Q,R)\in \{(1,11,5), (1,3,13),(1,1,5)\}.
\end{equation}

Next, we let $Q=1$ so that $W_1W_2=W_3$ from which it follows that
\begin{equation}\label{Eq:Pell}
  A^2-5\left(\frac{B+4}{5}\right)^2=-4.
\end{equation} It is well known \cite{Leveque, Lind} that the solutions to the Pell equation $X^2-5Y^2=-4$ are
\[(X,Y)=(\pm {\mathfrak L}_{2n-1},\pm {\mathfrak F}_{2n-1}),\] where ${\mathfrak L}_N$ and ${\mathfrak F}_N$ are, respectively, the $N$th Lucas and $N$th Fibonacci numbers, with ${\mathfrak L}_0=2$ and ${\mathfrak F}_0=0$. Thus, from \eqref{Eq:Pell}, we have
\[A=\pm {\mathfrak L}_{2n-1} \quad \mbox{and} \quad B=\pm 5{\mathfrak F}_{2n-1}-4.\]
Since
\[{\mathfrak L}_{2n-1}\equiv \left\{\begin{array}{cl}
  1 \pmod{4} \quad & \mbox{if and only if} \quad n\equiv 1 \pmod{3}\\[.5em]
  3 \pmod{4} \quad & \mbox{if and only if} \quad n\equiv 0 \pmod{3},
 \end{array} \right.\]
 \[{\mathfrak F}_{2n-1}\equiv \left\{\begin{array}{cl}
  1 \pmod{4} \quad & \mbox{if and only if} \quad n\equiv 0,1 \pmod{3}\\[.5em]
  3 \pmod{4} \quad & \mbox{is not possible,}
 \end{array} \right.\] and $A\equiv B\equiv 1 \pmod{4}$,
 the solutions to \eqref{Eq:Pell} are
 \begin{equation}\label{Eq:(A,B)}
 (A,B)=\left\{\begin{array}{cl}
   (-{\mathfrak L}_{2n-1},5{\mathfrak F}_{2n-1}-4) & \mbox{if $n\equiv 0\pmod{3}$}\\[.5em]
   ({\mathfrak L}_{2n-1},5{\mathfrak F}_{2n-1}-4) & \mbox{if $n\equiv 1\pmod{3}$}\\[.5em]
   \mbox{no solutions} & \mbox{if $n\equiv 2\pmod{3}$.}
 \end{array}\right.
 \end{equation}

 We provide details only in the case $n\equiv 1 \pmod{3}$, since the case $n\equiv 0 \pmod{3}$ is similar and yields no new solutions. Using \eqref{Eq:(A,B)}, Proposition \ref{Prop:LF} and basic properties of the recurrence relation for Lucas numbers, we deduce that
 \begin{align}\label{Eq:PRsquares}
 \begin{split}
   P&=B+2-2A={\mathfrak L}_{2(n-2)}-2={\mathfrak L}_{n-2}^2 \quad \mbox{if $n\equiv 1 \pmod{6}$,}\\
   R&=B+2+2A={\mathfrak L}_{2(n+1)}-2={\mathfrak L}_{n+1}^2 \quad \mbox{if $n\equiv 4 \pmod{6}$.}
   \end{split}
 \end{align}
Observe then that \eqref{Eq:PRsquares} contradicts the fact that $P$ and $R$ are squarefree unless $n=1$, in which case $A=B=1$ and $\FF_{2,A,B}(x)=\Phi_5(x)$. Note that $(P,Q,R)=(1,1,5)$, which has already been found in \eqref{Eq:P=1}.

Thus, we only need to analyze the other two solutions
\[(P,Q,R)\in \{(1,11,5),(1,3,13)\}\] from \eqref{Eq:P=1}. The corresponding coefficient pairs $(A,B)$ for $\FF_{2,A,B}(x)$ are, respectively, $(A,B)=(11,31)$ and $(A,B)=(9,19)$, which contradicts the fact that $A\equiv B\equiv 1\pmod{4}$. Consequently, the only solution found when $W_1>0$ and $W_2>0$ is $A=B=1$.

 We turn now to the situation when $W_1<0$ and $W_2<0$. Using Maple to solve the system \eqref{Eq:PQR} yields the equation
 \begin{equation}\label{Eq:W1<0W2<0}
 P^2Q^2-2PQ^2R+Q^2R^2+32PQ-16PR+32QR+256=0.
 \end{equation}
 Solving \eqref{Eq:W1<0W2<0} for $P$ gives
  \begin{equation}\label{Eq:Pcase2}
 P=\dfrac{Q^2R\pm 4\sqrt{R(Q^2+4)(R-4Q)}-16Q+8R}{Q^2}.
 \end{equation}
 Since $P$ and $R$ are positive integers, we see from \eqref{Eq:Pcase2} that we must have $R>4Q$. Moreover, since $R$ is odd and squarefree with $\gcd(R,Q)=1$, it follows that $R$ divides $Q^2+4$ since $R(Q^2+4)(R-4Q)$ must be a square. Then, making the observation that equation \eqref{Eq:W1<0W2<0} is symmetric in $P$ and $R$ (or simply solving \eqref{Eq:W1<0W2<0} for $R$), we deduce that $P>4Q$ and $P$ divides $Q^2+4$. Piecing together this information tells us, on the one hand, that  $PR>16Q^2$ and, on the other hand, that $PR$ divides $(Q^2+4)$, since $\gcd(P,R)=1$. That is, we conclude that
 \[16Q^2<PR\le Q^2+4,\] an obvious contradiction. Therefore, there are no additional solutions arising from the vacuous situation of $W_1<0$ and $W_2<0$, and the proof of item \eqref{Main1 I4} is complete.

For item \eqref{Main1 I5}, let $A=4t-4s^2-4s+1$. Then, we see that
\[\FF_{3,A,B}(x)=\left\{
\begin{array}{cl}
\FF_{2,2s+1,2t+1}(x)\FF_{2,-(2s+1),2t+1}(x) & \mbox{if $B=4t^2+4t-8s^2-8s+1$}\\[.5em]
\GG_{2,2s+1,2t+1}(x)\GG_{2,-(2s+1),2t+1}(x) & \mbox{if $B=4t^2+4t+8s^2+8s+5$,}
\end{array}\right.\]
where
\[\GG_{2,C,D}(x):=x^4+Cx^3+Dx^2-Cx+1.\]

Conversely, to derive these parametric values of $A$ and $B$, we assume that $\FF_{3,A,B}(x)=\FF_{2,A,B}(x^2)$ is reducible. Since $\FF_{2,A,B}(x)$ is irreducible over $\Q$, it follows that $\FF_{3,A,B}(\pm 1)=B+2+2A\ne 0$. Hence, $\FF_{3,A,B}(x)$ has no linear factors. If we assume that $\FF_{3,A,B}(x)$ has an irreducible quadratic factor, then Maple tells us in every viable situation that $\FF_{2,A,B}(x)$ is reducible, contradicting the fact that $\FF_{2,A,B}(x)$ is irreducible. Thus, we deduce that
\[\FF_{3,A,B}(x)=u_1(x)u_2(x),\]
where $u_i(x)=x^4+a_ix^3+b_ix^2+c_ix\pm 1$ is irreducible over $\Q$. Next, we expand $u_1(x)u_2(x)$, equate coefficients with $\FF_{3,A,B}(x)$ and use Maple to solve the resulting systems of equations, which ultimately yields the parametric values of $A$ and $B$ given here.

Next, for item \eqref{Main1 I6}, we recall notation from \eqref{Eq:Defs} and note then that
\[W_1=2-A,\quad W_2=3A+2 \quad \mbox{and} \quad W_3=A^2-4A+8.\] Since $A\equiv 1 \pmod{4}$, we can write $A=4t+1$ for some $t\in \Z$. Let
\[G(t):=(4t-1)(12t+5)(16t^2-8t+5).\] Since $G(t)\equiv 2(t+2)(t^2+t+2) \pmod{3}$, we see from Lemma \ref{Lem:ObstructionCheck} that we only have to check for obstructions at the prime $\ell=2$. Since $G(1)\equiv 3 \pmod{4}$, we conclude that $G(t)$ has no obstructions. Thus, by Corollary \ref{Cor:Squarefree}, we deduce that there are infinitely many primes $p$ such that $G(p)$ is squarefree. Let $p$ be such a prime, and let $A=4p+1$. It follows that $W_1$, $W_2$ and $W_3$ are squarefree. Also, $\FF_{3,A,A}(x)$ is irreducible over $\Q$ by item \eqref{Main1 I6} since $A\ne 1$.  Since $2\nmid W_1W_2W_3$, straightforward gcd calculations reveal that
\[P=Q=1\quad \mbox{and} \quad R\in \{1,5\}.\] Observe that $\abs{W_3}=1$ has no integer solutions. Furthermore, in integers,
\[\abs{W_1}=1 \ \mbox{if and only if} \ A\in \{1,3\}, \ \mbox{while} \ \abs{W_2}=1 \ \mbox{if and only if} \ A=-1.\]
It follows that
\begin{equation}\label{Eq:Wiconditions}
\mbox{none of } \ W_1,\ W_2,\ W_1W_2,\ W_1W_3,\ W_2W_3 \ \mbox{ and }\ W_1W_2W_3 \ \mbox{ is a square,}
\end{equation} except possibly when $W_2=W_3=\pm 5$. Since $W_2=W_3=-5$ has no integer solutions and $W_2=W_3=5$ has only the solution $A=1$, we conclude that
\eqref{Eq:Wiconditions} holds. Consequently, $\Gal(\FF_{3,A,A})\simeq C_2^2\wr C_2$ by Theorem \ref{Thm:AP}.

We turn now to item \eqref{Main1 I7}. For $n=3$, we see by item \eqref{Main1 I5} that
\begin{align}
4t-4s^2-4s+1&=4t^2+4t-8s^2-8s+1 \label{Eq:1}\\ \nonumber
&\mbox{or}\\
4t-4s^2-4s+1&=4t^2+4t+8s^2+8s+5. \label{Eq:2}
\end{align}
Solving \eqref{Eq:2} reveals no solutions, while solving \eqref{Eq:1} yields the solutions $(s,t)\in \{(0,0),(-1,0)\}$. Both of these solutions produce the value $A=1$, and it is easy to verify that $\FF_{3,1,1}(x)=\Phi_5(x)\Phi_{10}(x)$. It follows then that $\FF_{n,1,1}(x)$ is reducible over $\Q$ for all $n\ge 3$.

Alternatively, we can let $w(x):=\FF_{2,A,A}(x)$ in Proposition \ref{Prop:NG}. Then, using Maple, it is easy to verify that the only solution to $w(x)=S_0(x)^2-xS_1(x)^2$ is
\[S_0(x)=x^2+x+1 \quad \mbox{and} \quad S_1(x)=x+1,\] so that $A=1$, and that $w(x^2)=S_0(x)^2-xS_1(x)^2$ has no solutions (see the argument in \cite{JonesBAMS}). Hence, we conclude that $\FF_{n,A,A}(x)$ is reducible if and only if $A=1$.

 Finally, for item \eqref{Main1 I8}, we first note that $\FF_{3,1,1}(x)$ is not monogenic since $\FF_{3,1,1}(x)$ is reducible over $\Q$. Assume then that $A\ne 1$ so that $\FF_{3,A,A}(x)$ is irreducible over $\Q$ by item \eqref{Main1 I6}. Recall from \eqref{Eq:Delta(f)} that
\[\Delta(\FF_{3,A,A})=2^8W_1^2W_2^2W_3^4.\] Although $\FF_{2,A,A}(x)$ is monogenic by item \eqref{Main1 I2}, we use Theorem \ref{Thm:Dedekind} with $q=2$ and  $T(x):=\FF_{3,A,A}(x)$ to show in contrast that $\FF_{3,A,A}(x)$ is not monogenic.   Since $A\equiv 1 \pmod{4}$, we have that
\[\overline{T}(x)=\Phi_5(x)^2=(x^4+x^3+x^2+x+1)^2,\] and so we can let $h_1(x)=h_2(x)=\Phi_5(x)$. Then
\begin{align*}
  F(x)&=\dfrac{h_1(x)h_2(x)-T(x)}{2}\\
  &=\dfrac{\Phi_5(x)^2-\FF_{3,A,A}(x)}{2}\\
  &=x^7-\left(\dfrac{A-3}{2}\right)x^6+2x^5-\left(\dfrac{A-5}{2}\right)x^4+2x^3-\left(\dfrac{A-3}{2}\right)x^2+x,
  \end{align*}
which implies that
\[\overline{F}(x)=x^7+x^6+x^2+x=x(x+1)^2\Phi_5(x),\] since $A\equiv 1 \pmod{4}$.
Hence, $\gcd(\overline{F},\overline{h_1},\overline{h_2})\ne 1$, and consequently, by Theorem \ref{Thm:Dedekind}, we conclude that $\FF_{3,A,A}(x)$ is not monogenic. It then follows from Lemma \ref{Lem:Inductivemonogenicity} that $\FF_{n,A,A}(x)$ is not monogenic for all $n\ge 3$, which completes the proof of the theorem.
\end{proof}




\end{document}